\documentclass{amsart}
\usepackage{amssymb}
\usepackage{amscd}
\usepackage[latin1]{inputenc}
\usepackage{amsmath}
\usepackage{amsfonts}
\usepackage{latexsym}
\usepackage{verbatim}
\usepackage{graphicx}
\usepackage{epsfig}
\usepackage{color}
\usepackage[normalem]{ulem}

\newtheorem{theorem}{Theorem}[section]

\newtheorem{lemma}[theorem]{Lemma}

\newtheorem{proposition}[theorem]{Proposition}

\theoremstyle{definition}
\newtheorem{definition}[theorem]{Definition}

\newcommand{\sft}{SFT}

\newcommand{\zz}{{\mathbb Z}}
\newcommand{\nz}{{\mathbb N}}

\newcommand{\zt}{{\zz^2}}

\newcommand{\ag}{\ensuremath{{\mathcal A}} }

\numberwithin{equation}{section}

\begin{document}
\title[On intrinsic ergodicity of factors of $\mathbb{Z}^d$ subshifts]{On intrinsic ergodicity of factors of $\mathbb{Z}^d$ subshifts}

\begin{abstract}

It is well-known that any $\mathbb{Z}$ subshift with the specification property has the property that every factor is intrinsically ergodic, i.e., every factor has a unique factor of maximal entropy. In recent work, other $\mathbb{Z}$ subshifts have been shown to possess this property as well, including $\beta$-shifts and a class of $S$-gap shifts. We give two results that show that the situation for $\mathbb{Z}^d$ subshifts with $d >1 $ is quite different. First, for any $d>1$, we show that any $\mathbb{Z}^d$ subshift possessing a certain mixing property must have a factor with positive entropy which is not intrinsically ergodic. In particular, this shows that for $d>1$, $\mathbb{Z}^d$ subshifts with specification cannot have all factors intrinsically ergodic. We also give an example of a $\mathbb{Z}^2$ shift of finite type, introduced by Hochman, which is not even topologically mixing, but for which every positive entropy factor is intrinsically ergodic.

\end{abstract}

\date{}
\author{Kevin McGoff}
\address{Kevin McGoff\\
Department of Mathematics\\
Duke University\\
Durham, NC 27708-0320}
\email{mcgoff@math.duke.edu}
\urladdr{http://www.math.duke.edu/$\sim$mcgoff/}
\author{Ronnie Pavlov}
\address{Ronnie Pavlov\\
Department of Mathematics\\
University of Denver\\
2280 S. Vine St.\\
Denver, CO 80208}
\email{rpavlov@du.edu}
\urladdr{www.math.du.edu/$\sim$rpavlov/}
\thanks{}
\renewcommand{\subjclassname}{MSC 2010}
\maketitle

\section{Introduction}
\label{intro} 

The well-known Variational Principle relates the concepts of measure-theoretic and topological entropy for dynamical systems, stating that the topological entropy of any dynamical system is the supremum of the measure-theoretic entropies of all invariant measures on that system. In general, there may be no measures achieving that supremum, but if the system is expansive, then at least one such measure, called a measure of maximal entropy, must exist (\cite{misiurewicz}).

A topological dynamical system is said to be intrinsically ergodic (\cite{parry}, \cite{weiss}) if it has a unique measure of maximal entropy. It is well-known that for $\zz$ (i.e. one-dimensional) subshifts, strong enough topological mixing conditions imply intrinsic ergodicity; for instance, it was shown in \cite{bowen} that the specification property implies intrinsic ergodicity. The specification property is also clearly preserved under factor maps, which implies that for a $\zz$ subshift with specification, every factor is intrinsically ergodic. In particular, since every topologically mixing $\zz$ shift of finite type has the specification property, every such system is also intrinsically ergodic, along with all of its factors. These facts lead to a natural question (\cite{boyle}), asked by Thomsen, of whether every factor of a $\beta$-shift is intrinsically ergodic. This question was answered in the affirmative in \cite{CT}, where the authors gave a new sufficient condition for intrinsic ergodicity that is preserved under factor maps. Informally, their condition imposes specification on ``most words'' in the subshift, in some quantifiable way.

Strictly speaking, for a zero-entropy system, every invariant measure is trivially a measure of maximal entropy, and so for such systems intrinsic ergodicity is equivalent to the existence of a unique invariant measure, also known as unique ergodicity. Since intrinsic ergodicity of zero-entropy systems is therefore a somewhat degenerate case, in this paper we will focus on the question of whether all positive entropy factors of a subshift are intrinsically ergodic. This slight restriction of scope changes none of the context of the work described above, since all subshifts with specification and all subshifts treated in \cite{CT} (see Proposition 2.4 there) have positive entropy. 

In the current work, we study the class of $\mathbb{Z}^d$ subshifts ($d > 1$) for which every positive entropy factor is intrinsically ergodic, proving two results which are somewhat surprising given the results for $d = 1$ summarized above. The first is that any $\mathbb{Z}^d$ subshift with a certain topological mixing property (see Definition~\ref{dstardef}) must have a non-intrinsically ergodic factor, which is antithetical to the previously described results for $d = 1$.

\begin{theorem}\label{badfactor}
For any $d > 1$ and any $\mathbb{Z}^d$ subshift $X$ that has the D*-condition and does not consist of a single fixed point, there exists a factor map $\phi$ so that $h(\phi(X)) > 0$ and $\phi(X)$ is not intrinsically ergodic.
\end{theorem}

Our second main result shows that there do exist $\mathbb{Z}^d$ subshifts (in fact shifts of finite type) for which every positive entropy factor is intrinsically ergodic. The subshifts we consider are examples of Hochman (\cite{hochman}) and are not topologically mixing; in fact, they have a forced hierarchical structure similar to substitutionally defined SFTs in the literature (\cite{mozes}, \cite{robinson}).

\begin{theorem}\label{goodfactors}
There exist $\mathbb{Z}^2$ shifts of finite type with arbitrarily large entropy for which every factor with positive topological entropy is intrinsically ergodic.
\end{theorem}

\section{Definitions and preliminaries}
\label{defns}

Let $A$ denote a finite set, which we will refer to as an alphabet. 

\begin{definition}
A \textbf{pattern} over $A$ is a member of $A^S$ for some $S \subset \mathbb{Z}^d$, which is said to have \textbf{shape} $S$. For $d=1$ and $S$ an interval, patterns are generally called \textbf{words}.
\end{definition}

We only consider patterns to be defined up to translation, i.e., if $u \in A^S$ for a finite $S \subset \mathbb{Z}^d$ and $v \in A^T$, where $T = S+p$ for some $p \in \mathbb{Z}^d$, then we write $u = v$ to mean that $u(s) = v(s+p)$ for each $s$ in $S$.

For any patterns $v \in A^S$ and $w \in A^T$ with $S \cap T = \varnothing$, we define the concatenation $vw$ to be the pattern in $A^{S \cup T}$ defined by $(vw)(S) = v$ and $(vw)(T) = w$.

\begin{definition}
For any finite alphabet $A$, the \textbf{$\mathbb{Z}^d$-shift action} on $A^{\mathbb{Z}^d}$, denoted by $\{\sigma_t\}_{t \in \mathbb{Z}^d}$, is defined by $(\sigma_t x)(s) = x(s+t)$ for $s,t \in \mathbb{Z}^d$. 
\end{definition}

We always think of $A^{\mathbb{Z}^d}$ as being endowed with the product discrete topology, with respect to which it is obviously compact. 

\begin{definition}
A \textbf{$\mathbb{Z}^d$ subshift} is a closed subset of $A^{\mathbb{Z}^d}$ which is invariant under the $\mathbb{Z}^d$-shift action. A $\zz^d$ subshift is said to be \textbf{non-trivial} if it contains at least two points.
\end{definition}

\begin{definition} 
The \textbf{language} of a $\mathbb{Z}^d$ subshift $X$, denoted by $L(X)$, is the set of all patterns with finite shape which appear in points of $X$. For any finite $S \subset \mathbb{Z}^d$, let $L_S(X) := L(X) \cap A^S$, the set of patterns in the language of $X$ with shape $S$.
\end{definition}

Any subshift inherits a topology from $A^{\mathbb{Z}^d}$, with respect to which it is compact. Each $\sigma_t$ is a homeomorphism on any $\mathbb{Z}^d$ subshift, and so any $\mathbb{Z}^d$ subshift, when paired with the $\mathbb{Z}^d$-shift action, is a topological dynamical system. For a subshift $X$, we consider the set $\mathcal{M}(X)$ of all Borel probability measures on $X$ that are invariant under all shifts $\sigma_t$. Note that $\mathcal{M}(X)$ is compact in the weak$^*$ topology. For a measure $\mu$ in $\mathcal{M}(X)$ and a pattern $w$ in $L(X)$, we let $\mu(w) = \mu([w])$, where $[w]$ denotes the cylinder set defined by $w$. 

\begin{definition}
A $\mathbb{Z}^d$ subshift $X$ is called \textbf{uniquely ergodic} if $|\mathcal{M}(X)| = 1$, i.e., if there is only one invariant Borel probability measure on $X$.
\end{definition}

Any $\mathbb{Z}^d$ subshift can also be defined in terms of disallowed patterns: for any set $\mathcal{F}$ of patterns over $A$, one can define the set 
$$X(\mathcal{F}) := \{x \in A^{\mathbb{Z}^d} \ : \ (x)(S) \notin \mathcal{F} \ \text{ for all finite } S \subset \mathbb{Z}^d\}.$$It is well known that any $X(\mathcal{F})$ is a $\mathbb{Z}^d$ subshift, and all $\mathbb{Z}^d$ subshifts are representable in this way. 

\begin{definition}
A \textbf{$\mathbb{Z}^d$ shift of finite type (SFT)} is a $\mathbb{Z}^d$ subshift equal to $X(\mathcal{F})$ for some finite set $\mathcal{F}$ of forbidden patterns. 
\end{definition}

\begin{definition}
A (topological) \textbf{factor map} is any continuous shift-commuting map $\phi$ from a $\mathbb{Z}^d$ subshift $X$ onto a $\mathbb{Z}^d$ subshift $Y$. A bijective factor map is called a \textbf{topological conjugacy}.
\end{definition}

It is well-known that any factor map $\phi$ is a so-called sliding block code, i.e. there exists $n$ (called the \textbf{radius} of $\phi$) so that $x(v + [-n,n]^2)$ uniquely determines $(\phi(x))(v)$ for any $x \in X$ and $v \in \mathbb{Z}^2$. (See \cite{LM} for a proof for $d = 1$, which extends to $d > 1$ without changes.) 
A factor map $\phi$ is \textbf{$1$-block} if it has radius $0$.

\begin{definition}\label{entdef}
The \textbf{topological entropy} of a $\zz^d$ subshift $X$ is
\[
h(X) := \lim_{n_1, \ldots, n_d \rightarrow \infty} \frac{1}{\prod_{i=1}^d n_i} \log |L_{\prod_{i=1}^d [1,n_i]}(X)|.
\]
\end{definition}

\begin{definition} \label{GeneralEntropyDef}
Let $\Omega$ be a finite set, and let $\mu$ be a probability measure on $\Omega$. Then the \textbf{entropy of $\mu$} is defined as
\begin{equation*}
H(\mu) = \sum_{\omega \in \Omega} -  \mu(\{\omega\}) \log \mu ( \{ \omega \}).
\end{equation*}
\end{definition}

We will make use of the following basic facts about entropy, which we present without proof. See~\cite{CoverThomas} for details.

\begin{proposition} \label{EntropyCardUB}
Suppose $\Omega$ is a finite set and $\mu$ is a probability measure on $\Omega$. Then
\begin{equation*}
H(\mu) \leq \log |\Omega|,
\end{equation*}
with equality if and only if $\mu$ is uniformly distributed on $\Omega$.
\end{proposition}

\begin{proposition} \label{EntropyProdUB}
Suppose $A$ and $\mathcal{I}$ are finite sets and $\mu$ is a probability measure on $A^{\mathcal{I}}$. For $C \subset \mathcal{I}$, let $\mu_C$ denote the projection (marginal) of $\mu$ onto $A^C$. Then for any partition $\mathcal{P}$ of $\mathcal{I}$, it holds that
\begin{equation*}
H(\mu) \leq \sum_{C \in \mathcal{P}} H(\mu_C).
\end{equation*}
\end{proposition}

\begin{definition} \label{MTentdef}
For a subshift $X$ and a measure $\mu$ in $\mathcal{M}(X)$, define $H_{\mu,N}$ to be the entropy of $\mu$ with respect to the partition given by $L_{[1,N]^d}(X)$: 
\begin{align*}
 H_{\mu,N}  = \sum_{w \in L_{[1,N]^d}(X)} - \mu(w) \log \mu(w).
\end{align*}
Then the \textbf{entropy} of the measure $\mu$ is given by
\begin{align*}
 h(\mu)  = \lim_{N \to \infty} \frac{1}{N^d} H_{\mu,N} = \inf_{N \to \infty} \frac{1}{N^d} H_{\mu,N}.
\end{align*}
\end{definition}

For any subshift $X$, the Variational Principle states that the topological entropy of $X$ is the supremum of the measure-theoretic entropies $h(\mu)$ over all $\mu \in \mathcal{M}(X)$, which motivates the following definition.

\begin{definition}
For any subshift $X$, any measure $\mu \in \mathcal{M}(X)$ for which $h(\mu) = h(X)$ is called a \textbf{measure of maximal entropy} for $X$.
\end{definition}

For general topological systems the supremum $h(X)$ may not be achieved; nonetheless, every subshift has at least one measure of maximal entropy; see \cite{misiurewicz} for a proof. It is natural to wonder when a subshift has a single such measure, which motivates the following definition (\cite{parry}, \cite{weiss}).

\begin{definition}  
A subshift $X$ is said to be \textbf{intrinsically ergodic} if it has exactly one measure of maximal entropy.
\end{definition}

We now turn to the mixing condition that appears in Theorem \ref{badfactor}.

\begin{definition}\label{dstardef}
A $\mathbb{Z}^d$ subshift $X$ has the \textbf{D*-condition} if for any $n$ there exists $k_n$ with the property that for any $x,y \in X$, there exists $z \in X$ such that $z([-n,n]^d) = x([-n,n]^d)$ and $z(\mathbb{Z}^d \setminus [-(k_n + n), k_n + n]^d) = y(\mathbb{Z}^d \setminus [-(k_n + n), k_n + n]^d)$. 
\end{definition}

The D*-condition was defined in \cite{ruelle} as a property of subshifts which guarantees that any finite-range Gibbs measure on $X$ must be fully supported. 
It is significantly weaker than so-called uniform mixing conditions such as the uniform filling property and strong irreducibility/specification (see \cite{BPS}). The following fact follows almost immediately from Definition~\ref{dstardef}, but it will be expeditious to state it as a lemma.

\begin{lemma}\label{dstarplacement}
If $X$ has the D*-condition, $n \in \mathbb{N}$, and $k_n$ is defined as in Definition~\ref{dstardef}, then for any $x \in X$, any (possibly infinite) collection $\{v_j\}_{j \in J} \subset \mathbb{Z}^d$ such that the sets $v_j + [-(n+k_n),n+k_n]^d$ are disjoint for $j \in J$, and any $\{w_j\}_{j \in J} \subseteq L_{[-n,n]^d}(X)$, there exists $z \in X$ so that for any $j \in J$, it holds that $z(v_j + [-n,n]^d) = w_j$, and $z(t) = x(t)$ for any $\displaystyle t \notin \bigcup_{j \in J} v_j + [-(n+k_n), n+k_n]^2$. 
\end{lemma}

\begin{proof}

For finite $J$, we prove the lemma by induction on $|J|$. The case $|J| = 1$ is just the definition of the D*-condition. Now, suppose the lemma holds for $|J| = j$. Consider any $|J| = j+1$ and $\{v_j\}$ and $\{w_j\}$ as in the lemma, and choose any $j_0 \in J$. Then, one can first apply the inductive hypothesis for $J \setminus \{j_0\}$ to get $x' \in X$ which satisfies the conclusion of the lemma for all $j \in J \setminus \{j_0\}$. But then, applying the $|J| = 1$ case to $x'$ and $w_{j_0}$ yields $x$ satisfying the conclusion of the lemma for $J$ itself, completing the proof. 

Now, for infinite (but by necessity countable) $J$, we first assume $J$ to be $\mathbb{N}$ without loss of generality. Then, for each $m$, by appeal to the finite case, there exists $x_m$ which has the desired properties for $J_m = [1,m]$. The sequence $x_m$ has a convergent subsequence by compactness, and its limit has the desired properties for all $j$, completing the proof.

\end{proof}


In the proof of Theorem \ref{badfactor}, we will use the following technical lemma.

\begin{lemma}\label{dstarremoval}
If $X$ is a nontrivial $\mathbb{Z}^d$ subshift with the D*-condition, then there exists a pattern $w \in L(X)$ so that if we define $X'$ to be the subshift consisting of points in $X$ with no occurrences of $w$, then $h(X') > 0$.
\end{lemma}

\begin{proof}

Assume that $X$ is such a subshift. Since $X$ is nontrivial, the alphabet of $X$ contains at least $2$ letters. If $|L_S(X)| = 2$ for every finite shape $S$, then $X$ is a periodic orbit of two points, which does not have the D*-condition, a contradiction. Therefore, there exists $S$ so that $|L_S(X)| \geq 3$, and since enlarging $S$ cannot decrease $|L_S(X)|$, we assume without loss of generality that $S = [-n,n]^d$ for some $n$. Denote by $k = k_n$ the $k_n$ guaranteed by Definition~\ref{dstardef}, and choose any distinct patterns $t,u,v \in L_{[-n,n]^d}(X)$.

Begin with an arbitrary point $x \in X$, and use Lemma~\ref{dstarplacement} with the set $\{v_j\}_{j \in J} = \{0, 2k+2n+1, 2(2k+2n+1), \ldots, (2k+2n)(2k+2n+1)\}^d$ and $w_j = t$ for every $j$. In other words, we create $x' \in X$ with a finite equispaced grid of occurrences of $t$, whose centers have separation $2k + 2n + 1$ along each cardinal direction. Define $w = x([-n,(2k+2n)(2k+2n+1)+n]^d)$, a pattern in $L(X)$ which also contains the entire grid of occurrences of $t$ just described. 

We claim that if $X'$ is defined as in the lemma, then $h(X') > 0$.
To see this, we define a family of points in $X'$ in the following way: start with $x \in X$, and use Lemma~\ref{dstarplacement} with the set $\{v_j\}_{j \in J} = ((2k+2n)\mathbb{Z})^d$ and any choice of $w_j \in \{u,v\}$ for every $j$. In other words, we create points with an infinite equispaced grid filled with independent choices of $u$ or $v$, whose centers have separation $2k + 2n$ along each cardinal direction. We claim that all such points are in $X'$. Suppose for a contradiction that such a point, call it $y$, has an occurrence of $w$. However, note that no matter how $w$ is shifted, it will contain some occurrence of $t$ with center in $((2k+2n)\mathbb{Z})^d$ (because $w$ contained occurrences of $t$ in every coset in $(\mathbb{Z}^d)/((2k+2n)\mathbb{Z})^d$.) This occurrence gives a contradiction, since every translate of $[-n,n]^2$ with center in $((2k+2n)\mathbb{Z})^d$ in $y$ is filled with either $u$ or $v$, and so not $t$. For every $m$, this yields at least $2^{m^d}$ patterns in $L_{[-n, (m-1)(2k+2n) + n]^d}(X')$ (from $m^d$ independent choices of $u$ or $v$), and so
\[
h(X') \geq \lim_{m \rightarrow \infty} \frac{\log 2^{m^d}}{((m-1)(2k+2n) + 2n + 1)^d} > 0.
\]

\end{proof}

\section{Proof of Theorem~\ref{badfactor}}
\label{badproof}

The ``bad factor'' proving Theorem~\ref{badfactor} will always be of the same type; it will be a shift of finite type based on the lattice Widom-Rowlinson model from statistical physics (\cite{WR}). We first define this SFT.

\begin{definition}\label{WRdef}
For any $R_1 \leq R_2$, the $\mathbb{Z}^d$ \textbf{Widom-Rowlinson SFT with interaction distances $R_1$ and $R_2$}, denoted by $W_{R_1, R_2}$, is the SFT with alphabet $\{0,+,-\}$ which consists of all $x \in \{0,+,-\}^{\mathbb{Z}^d}$ satisfying the following local rules: (here and elsewhere, the distance between sites of $\mathbb{Z}^d$ always refers to the $\ell_{\infty}$ metric)\\

\noindent
$\bullet$ any pair of nonzero symbols must have distance greater than $R_1$

\noindent
$\bullet$ any pair of nonzero symbols with opposite signs must have distance greater than $R_2$

\end{definition}

It seems ``well-known'' that if $R_2$ is large compared to $R_1$, then $W_{R_1, R_2}$ is not intrinsically ergodic (see \cite{LG}). However, technically the cited paper only treats the case where $R_1 = 1$, and so we present a self-contained proof here. We will use a fairly standard Peierls argument, following \cite{LG}. 

\begin{theorem}\label{WRnotIE}
For any $d > 1$ and $R_1$, there exists $N = N(R_1,d)$ so that if $R_2 \geq N$, then $W_{R_1, R_2}$ is not intrinsically ergodic.
\end{theorem}

\begin{proof} 

Fix any $R_1$. For technical reasons, assume that $R_2 > 2^{5d+2} R_1^{2d}$. For brevity, we will refer to $W_{R_1, R_2}$ simply as $W$. For any $k$ a multiple of $R_1 + 1$, consider patterns on the cube $[-k,k]^d$ with boundary condition $\delta_{k,+}$ on $[-k,k]^d \setminus [-k+1,k-1]^d$ given by $\delta_{k,+}(v) = +$ if exactly one of the $v_i$ is $\pm k$ and all others are divisible by $R_1 + 1$, and $\delta_{k,+}(v) = 0$ otherwise. In other words, $\delta_{k,+}$ contains equispaced $+$ symbols at distance $R_1 + 1$ on each face of the boundary and $0$ symbols elsewhere on the boundary. This leaves only sites on $[-k+1,k-1]^d$ undefined, and so we can define the measure $\mu_{k,+}$ on $L_{[-k+1,k-1]^d}(X)$ which gives equal measure to every $x$ so that $x\delta_{k,+} \in L(W)$. For any $v \in [-k+1,k-1]^d$, we define $E_{-,v}$ to be the event that (i.e., set of patterns such that) there is a $-$ symbol at $v$. We will give an upper bound on $\mu_{k,+}(E_{-,v})$.

To this end, fix $v$ and consider any $x \in E_{-,v}$. Then, consider the union $U$ of $t + [-R_2,R_2]^d$ over all $t$ at which $(x\delta_{k,+})(t) = -$. This union is nonempty since $x \in E_{-,v}$ implies that $v \in U$. It is contained within $[-k+1,k-1]^d$ by the boundary condition $\delta_{k,+}$ and the fact that $-$ and $+$ symbols must be separated by distance greater than $R_2$. It may consist of several disjoint connected components; define $A$ to be the one containing $v$. Let $C$ be the ``outermost contour'' of $A$, i.e. the set of sites in $A^c$ that are adjacent to a site in $A$ but also can be connected to the boundary of $[-k,k]^d$ by a path of adjacent sites in $A^c$. 

Then define by $M$ (for ``moat'') the set of sites in $A$ within distance $R_2$ of $C$. It should be clear that every site in $M$ must be labeled by $0$ in $x$; such a site can't be a $-$ since it's within distance $R_2$ of a site not in $A$, and it can't be a $+$ since it's in $U$ and thereby within distance $R_2$ of a $-$ symbol. We note for future reference that $M$ is in fact determined by $C$, because of the following alternate definition of $M$: $M$ is the set of all sites within distance $R_2$ of $C$ that are ``inside'' $C$, i.e. which cannot be connected to a site on the boundary of $[-k,k]^d$ without passing through a site in $C$. We leave it to the reader to verify that this definition of $M$ is equivalent to the original one. See Figure~\ref{IEpic1} for an illustration. We note also, as it will be useful later, that $M$ has ``thickness'' at least $R_2$ in every cardinal direction, i.e. any line segment in a cardinal direction connecting a site inside $M$ to a site outside $M$ passes through at least $R_2$ consecutive sites of $M$ in between.

\begin{figure}[ht]
\centering
\includegraphics[scale=0.4]{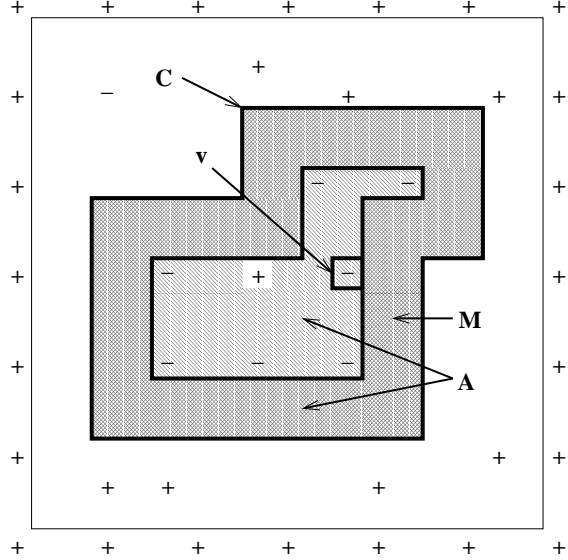}
\caption{A pattern $x$ in $E_{v,-}$ (the blank area represents $0$ symbols)}
\label{IEpic1}
\end{figure}

Define $E_{M,-,v}$ to be the set of all $x \in E_{-,v}$ which have a particular set $M$ as its ``moat.'' Clearly $E_{-,v}$ is the disjoint union of $E_{M,-,v}$ over all possible $M$. We wish to give an upper bound on each $\mu_{k,+}(E_{M,-,v})$ by a simple counting argument. Fix an $M$ and corresponding $E_{M,-,v}$, and define a function $\rho: E_{M,-,v} \rightarrow L_{[-k+1,k-1]^d}(W)$ as follows: $\rho(x)$ is obtained from $x$ by ``flipping'' (i.e. changing to $+$) every $-$ symbol at a location $a \in A$ with the following property: there exists a finite path $a = a_0, a_1, \ldots, a_m$ of sites in $A$ where $x(a_i) = -$ for every $i$, $a_m$ is within distance $R_2$ of $M$, and $a_i$ is within distance $R_2$ of $a_{i+1}$ for every $i$. The case $m = 0$ is included, i.e. $-$ symbols in $A$ which are themselves within distance $R_2$ from $M$ are flipped. (Figure~\ref{IEpic2} shows the application of $\rho$ to the pattern from Figure~\ref{IEpic1}.)


We first wish to show that $\rho(x)$ is indeed legal. Since changing $x$ to $\rho(x)$ involves only switching of $-$ symbols to $+$ symbols, the only possible problem would be if there exist $s,t$ with distance less than or equal to $R_2$ for which $x(s) = x(t) = -$ and one was flipped in the process of changing $x$ to $\rho(x)$ while the other was not. We show that such $s,t$ can not exist by considering three cases. First, it's not possible to have $s \in A$ and $t \notin A$: by definition of $M$, if $s$ and $t$ have distance less than or equal to $R_2$, then $s \in M$, which would imply $x(s) = 0$. (Clearly, the same proof shows that $s \notin A$, $t \in A$ is impossible.) Second, it's clearly not possible to have $s,t \notin A$, since then neither site would be flipped. The third case is $s,t \in A$, but the rules defining $\rho(x)$ imply that if a $-$ symbol in $A$ is within $R_2$ of another flipped $-$ symbol in $A$, then the first $-$ symbol must also be flipped, ruling this case out as well. 

\begin{figure}[ht]
\centering
\includegraphics[scale=0.4]{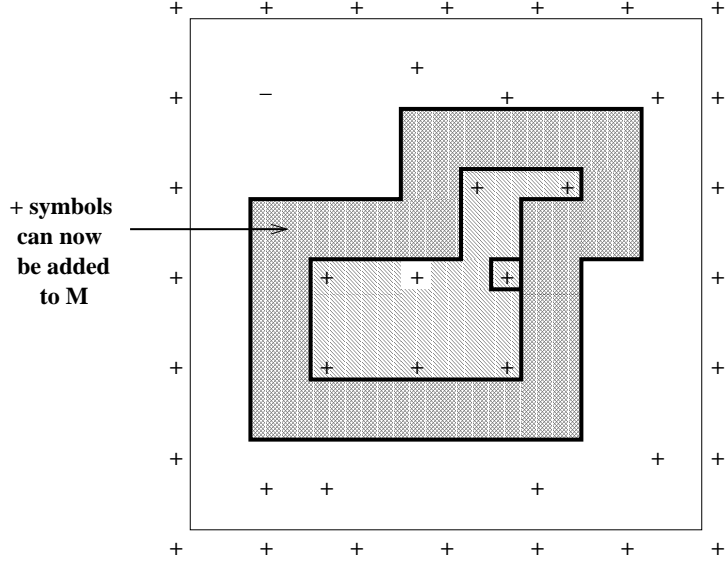}
\caption{$\rho(x)$ for the $x$ from Figure~\ref{IEpic1}}
\label{IEpic2}
\end{figure}

The map $\rho$ is not necessarily one-to-one on $E_{M,-,v}$; in looking at $\rho(x)$, if one sees a $+$ symbol, it is not immediately clear whether that $+$ was a $+$ present in $x$ or a $-$ changed to a $+$. In order to determine $x$ from $\rho(x)$, it would suffice to know whether each such $+$ was flipped or not. We first note that it's sufficient to know whether the $+$ symbols within distance $R_2$ of a site in $M$ were flipped or not. To see this, note that the only $+$ symbols in $\rho(x)$ which could have been flipped must be connected to a $+$ symbol within distance $R_2$ of $M$ by a path of $+$ symbols of distances at most $R_2$. But then, either all $+$ symbols on the path were flipped or all were not flipped, since their distances of at most $R_2$ mean that they were all forced to have the same sign in $x$. This implies that the set of all $+$ symbols in $\rho(x)$ which were flipped is precisely the set of $+$ symbols which can be connected to a flipped $+$ symbol within distance $R_2$ of $M$ by a path of $+$ symbols of distances at most $R_2$.

We now wish to give an upper bound on the number of ways in which the $+$ symbols within distance $R_2$ of $M$ could have arisen. To do this, consider $M'$, the set of sites in $A$ within $R_2$ of a site in $M$. This is a subset of $M''$, the set of sites in $A$ within $2R_2$ of a site in $C$. Then $M''$ can be written as a union of sets $A \cap (c + [-2R_2, 2R_2]^d)$ over all $c \in C$. We break each such set into $4^d$ disjoint regions $A \cap (c + \prod_{j=1}^d [i_j R_2, (i_j+1)R_2])$ for $-2 \leq i_j < 2$. Then, for each one of these regions, all $+$ symbols inside must either all have been flipped or all have been not flipped, since the diameter of the region is at most $R_2$. This means that we have an upper bound of $2^{4^d}$ on the number of ways in which each $+$ symbol in $A \cap (c + [-2R_2, 2R_2]^d)$ can have status ``flipped'' or ``not flipped,'' yielding the following upper bound:
\begin{equation} \label{Highlands}
|\rho^{-1}(y)| \leq 2^{4^d |C|}.
\end{equation}

For every $\rho(x)$, we wish to generate many legal patterns in $W$ by changing some of the $0$ symbols in $M$ to $+$ symbols. For this purpose, we note that in 
$\rho(x)$, no site in $M$ is within $R_2$ of a $-$ symbol; any such symbol would have been flipped by definition of $\rho$. Therefore, when introducing $+$ symbols into sites in $M$ in $\rho(x)$, we must only check that we do not create a pair of $+$ symbols with distance less than $R_1$.

By definition of $M$, for every site $t \in C$, there exists a direction $u \in \{\pm e_i\}$ so that $t + ku \in M$ for all $1 \leq k \leq R_2$. Choose a fixed $u$ so that there is a set $B \subseteq C$, $|B| \geq |C|/2d$, for which each site in $B$ satisfies the above condition for $u$. Define $B' = \bigcup_{b \in B, 1 \leq k \leq R_2} \{b + k u\}$; clearly this union is disjoint and $|B'| = R_2 |B| \geq R_2 |C|/2d$. Now we use a greedy algorithm to choose a subset $B'' \subseteq B'$ so that each pair of sites in $B''$ is separated by distance more than $R_1$. Formally speaking, start with $B'' = \varnothing$, and add sites to $B''$ in the following way. Choose any site $b' \in B'$, remove it from $B'$, and add it to $B''$ (making it the only element of $B''$ for the moment). Then, remove all sites in $B'$ within distance $R_1$ of $b'$. Repeat this procedure until $B'$ is empty. At each step of this procedure, we increase $B''$ by exactly one and decrease $B'$ by less than 
$4^d R_1^d$, and so $|B''| \geq |B'|/(4^d R_1^d) \geq R_2 |C| / (8^d R_1^d)$.


Finally, we wish to remove from $B''$ any sites within distance $R_1$ of $C$, which reduces the size of $B''$ by less than or equal to $R_1^d |C|$. Doing this yields a set $B'''$ with 
\begin{equation} \label{Linger}
|B'''| \geq |C| \left(\frac{R_2}{8^d R_1^d} - R_1^d\right) \geq |C| \frac{R_2}{2^{3d+1} R_1^d},
\end{equation}
since $R_2 > 2^{3d+1} R_1^{2d}$.

We finally note that in $\rho(x)$, since all sites in $B'''$ are separated by more than $R_1$ from each other and none is within $R_1$ of any site outside $M$, we may independently change the $0$s at sites in $B'''$ to $+$ in any way to yield a legal pattern in $W$. For $\rho(x) \neq \rho(x')$ given by $x,x' \in E_{M,-,v}$, the sets of patterns thus obtained will obviously be disjoint, since the only changes are made within $M$, where $\rho(x)$ and $\rho(x')$ both were labeled with all $0$ symbols. Therefore, by (\ref{Highlands}) and (\ref{Linger}), we have
\begin{multline*} 
\big|L_{[-k,k]^2}(W) \cap [\delta_{k,+}] \big| \geq 2^{|B'''|} \big|\rho(E_{M,-,v})\big| \geq 2^{|C| \frac{R_2}{2^{3d+1} R_1^d}} \big|\rho(E_{M,-,v})\big| \geq \\ 2^{|C| (\frac{R_2}{2^{3d+1} R_1^d} - 4^d)} |E_{M,-,v}| \geq 
2^{|C| \frac{R_2}{2^{3d+2} R_1^d}} |E_{M,-,v}|.
\end{multline*}

(The last inequality holds since $R_2 > 2^{5d+2} R_1^d$.) This inequality gives that $\mu_{k,+}(E_{M,-,v}) \leq 2^{-|C| \frac{R_2}{2^{3d+2} R_1^d}}$. Then, since $C$ determines $M$, we obtain that
\[
\mu_{k,+}(E_{-,v}) \leq \sum_{C} 2^{-|C| \frac{R_2}{2^{3d+2} R_1^d}} \leq \sum_{n \geq 1} C_n \Big(2^{\frac{R_2}{2^{3d+2} R_1^d}}\Big)^{-n},
\]
where $C_n$ is the number of possible contours of size $n$ surrounding the origin. It is well-known that up to translation, the number of connected subsets of sites of $\mathbb{Z}^d$ with size $n$ (the so-called lattice animals) is bounded from above by $(2^{2d-1})^n$ \cite{WS}, and then the number of translates of a contour that could surround the origin is bounded from above by $(2n)^d < (2^d)^n$, so that we have $C_n \leq (2^{3d-1})^n$. Therefore,

\[
\mu_{k,+}(E_{-,v}) \leq \sum_{n \geq 1} \Big(2^{\frac{R_2}{2^{3d+2} R_1^d} - (3d-1)}\Big)^{-n} = \frac{\alpha}{1 - \alpha},
\]
where $\alpha = 2^{-\Big(\frac{R_2}{2^{3d+2} R_1^d} - (3d - 1)\Big)}$. (We note that $\alpha < 1$ due to the original assumption $R_2 > 2^{5d+2} R_1^{2d}$.) This bound holds for every $k$ that is a sufficiently large multiple of $R_1 + 1$ and every $v \in [-k+1,k-1]^d$. Note that if $R_1$ is fixed and we allow $R_2$ to approach infinity, then $\alpha \rightarrow 0$, yielding an upper bound approaching $0$ on $\mu_{k,+}(E_{-,v})$ that does not depend on $k$ or $v$. On the other hand, we claim that $h(W) \geq \frac{\log 2}{(R_1 + 1)^d} > 0$ regardless of how large $R_2$ is; indeed, one can make legal patterns in $W$ by independently choosing sites with all coordinates divisible by $R_1 + 1$ to be $0$ or $-$ and assigning all other sites to be $0$. Now, note that for any measure $\mu$ in $\mathcal{M}(W)$, if we define $\mu(+) = \beta_1$, $\mu(-) = \beta_2$, and $\beta = \beta_1 + \beta_2$, then
\begin{multline*}
h(\mu) \leq H_{\mu,1} =  -(1 - \beta) \log(1 - \beta) - \beta_1 \log \beta_1 - \beta_2 \log \beta_2 \leq -(1 - \beta) \log (1 - \beta)\\ 
- (\beta/2) \log(\beta/2) - (\beta/2) \log(\beta/2) = -(1 - \beta) \log (1 - \beta) - \beta \log(\beta/2)
\end{multline*}
by convexity. The right-most expression in the above display is clearly continuous in $\beta$ and decreases to $0$ when $\beta$ decreases to $0$, and so we can choose $\beta' > 0$ so that  $-(1-\beta') \log (1-\beta') - \beta' \log (\beta'/2) = \frac{\log 2}{(R_1 + 1)^d}$. Then, for any measure of maximal entropy $\mu$ of $W$, we must have $\mu(+ \cup -) > \beta'$ (otherwise, by the above computations, $h(\mu) < \frac{\log 2}{(R_1 + 1)^d} \leq h(W)$, contradicting $\mu$ being a measure of maximal entropy on $W$).

Now define $N$ so that whenever $R_2 \geq N$, we have that $\mu_{k,+}(E_{-,v}) < \frac{\beta'}{3}$ for every $k$ that is a sufficiently large multiple of $R_1 + 1$ and every $v \in [-k+1,k-1]^d$. Then, take any weak$^*$ limit point $\mu_+$ of the measures $\displaystyle \frac{1}{|[-k/2,k/2]^d|} \sum_{v \in [-k/2,k/2]^d} \sigma_v \mu_{k,+}$ as $k \rightarrow \infty$; clearly $\mu_+$ is shift-invariant. We note that for any two patterns in $L(W)$ separated by distance greater than $R_2$, the remainder of $\mathbb{Z}^d$ can be filled with $0$s to make a point in $W$, which implies that $W$ is strongly irreducible as defined in \cite{BS}. Therefore, by Proposition 1.12(ii) from that same paper, $\mu_+$ is a measure of maximal entropy for $W$. However, $\mu_+(E_{-,0})$ is a limit of averages of $\mu_{k,+}(E_{-,v})$ over various $v$ and is therefore less than or equal to $\frac{\beta'}{3}$ whenever $R_2 > N$. We finally note that by the symmetry of the local rules defining $W$, there must be another measure of maximal entropy $\mu_-$ obtained by simply flipping signs of $+$ and $-$ symbols for $\mu_+$; more rigorously, for any pattern $w$, let $\mu_-(w) := \mu_+(\overline{w})$, where $\overline{w}$ is obtained from $w$ by flipping every nonzero symbol. Clearly, when $R_2 > N$, $\mu_-(E_{-,0}) = \mu_+(E_{+,0}) \geq \beta' - \mu_+(E_{-,0}) \geq 2\beta'/3$, proving that $\mu_+ \neq \mu_-$. Hence $W$ is not intrinsically ergodic.

\end{proof}

We now show that any $\mathbb{Z}^d$ subshift with the D*-condition has Widom-Rowlinson SFTs as factors.

\begin{theorem}\label{WRfactor}
For any $d > 1$ and any non-trivial $\mathbb{Z}^d$ subshift $X$ which has the D*-condition, there exists $N'$ so that for any
$R_2 \geq R_1 \geq N'$, there is a factor map $\phi : X \to W_{R_1, R_2}$.
\end{theorem}

\begin{proof}

Suppose that $X$ is a nontrivial $\zz^d$ subshift with the D*-condition, with associated sequence $(k_n)_{n \in \mathbb{N}}$ as in Definition~\ref{dstardef}. We may clearly assume without loss of generality that $(k_n)$ is nondecreasing, since replacing any $k_n$ with a larger integer preserves the conclusion of Definition~\ref{dstardef}. First, since $h(X) > 0$, by Lemma~\ref{dstarremoval}, we may choose $n \in \mathbb{N}$ and a pattern $w \in L_{[1,n]^d}(X)$ so that removing $w$ from $L(X)$ yields a nonempty subshift $X'$ with $h(X') > 0$. 

Now we'll use some results from \cite{hochman} to create some markers. In \cite{hochman}, for any $m > k$, a \textbf{marker} on $X'$ is defined to be any pattern $M \in L_{[-m,m]^d \setminus [-(m-k),m-k]^d}(X')$ with the property that if $x \in X'$ and $x([-m,m]^d \setminus [-(m-k),m-k]^d) = x(v + ([-m,m]^d \setminus [-(m-k),m-k]^d)) = M$ for some $v \in [-(2m-k+1), 2m-k+1]^d$, then $v = 0$. Informally, a marker is a word on a square annulus such that two copies may overlap, but not in such a way that one intersects the ``interior'' of the other. Note that $2m - k + 1 > k$, and so any marker also satisfies this definition when $[-(2m-k+1), 2m-k+1]^d$ is replaced by $[-k,k]^d$. Since $h(X') > 0$, Proposition 3.5 from \cite{hochman} guarantees, for any large enough $k$, the existence of a marker $M$ for some $m > k$ which can be completed to at least two patterns in $L_{[-m,m]^d}(X')$ (in fact it guarantees much more, but this is all that we'll need). We apply this to $k = 2n + 4k_n$, yielding patterns $w_+, w_- \in L_{[-m,m]^d}(X')$ (the completions of the marker $M$ to $[-m,m]^d$) with the following property: if $x \in X'$ and $x([-m,m]^d)$ and $x(v + [-m,m]^d)$ are both in $\{w_+, w_-\}$ for some $v \in [-(2n+4k_n),2n+4k_n]^d$, then $v = 0$. Clearly this implies that $m > 2n + 4k_n$ and that $w_+$ and $w_-$ contain no occurrences of $w$.

Now we claim that $N' = 2k_m + 2m + 1$ suffices to prove the theorem. To this end, choose any $R_1$ and $R_2$ where $R_2 \geq R_1 > 2k_m + 2m$. We define our factor map $\phi$ on $X$ as follows: if $x(v + [-m,m]^d) \notin \{w_+, w_-\}$, then $(\phi(x))(v) = 0$. If $x(v + [-m,m]^d) \in \{w_+, w_-\}$ and there exists $u \in [-R_1, R_1]^d$ so that $x(v + u + [-m,m]^d) \in \{w_+, w_-\}$, then $(\phi(x))(v) = 0$. If $x(v + [-m,m]^d) \in \{w_+, w_-\}$ and there exists $u \in [-R_2, R_2]^d$ so that 
$x(v + u + [-m,m]^d) \in \{w_+, w_-\}$ and $x(v + u + [-m,m]^d) \neq x(v + [-m,m]^d)$, then $(\phi(x))(v) = 0$. If $x(v + [-m,m]^d) \in \{w_+, w_-\}$ and none of the previous three rules applies, then $(\phi(x))(v) = +$ if $x(v + [-m,m]^d) = w_+$ and $(\phi(x))(v) = -$ if $x(v + [-m,m]^d) = w_-$. The reader may check that these rules are not contradictory, and so $\phi$ is a continuous shift-commuting map on $X$. It remains to show that $\phi$ is surjective, i.e. that $\phi(X) = W_{R_1, R_2}$. 

\begin{figure}[ht]
\centering
\includegraphics[scale=0.35]{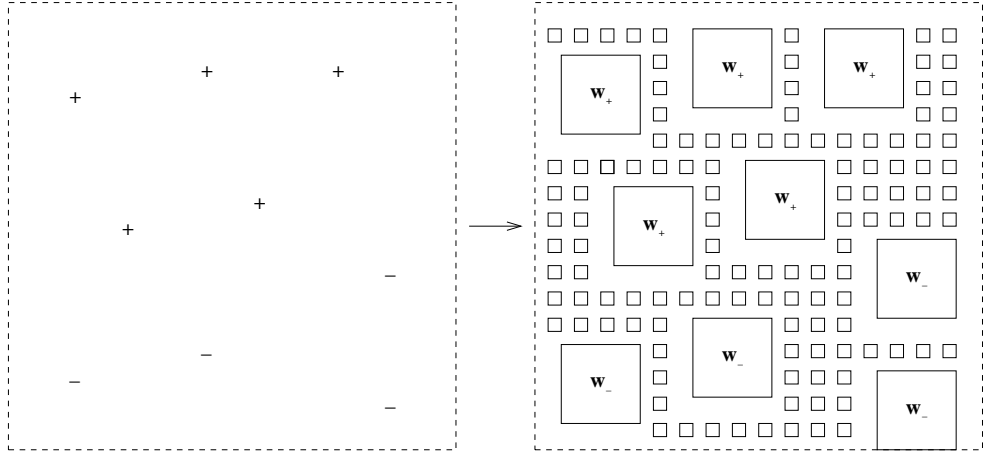}
\caption{Constructing $\phi$-preimage of a point of $W$ (smallest boxes are occurrences of $w$)}
\label{IEpic3}
\end{figure}

It is easy to see that $\phi(X) \subseteq W_{R_1, R_2}$, since the rules defining $\phi$ force any $\phi(x)$ to satisfy the local rules from Definition~\ref{WRdef}. We now prove the opposite inclusion. Choose any $y \in W_{R_1, R_2}$, and we will construct $x \in X$ so that $\phi(x) = y$ (see Figure~\ref{IEpic3}). We begin with an arbitrary $x' \in X$. We will use Lemma~\ref{dstarplacement} to change letters on $x'$ in several phases, eventually yielding the desired $x$. We begin by defining $x(v + [-m,m]^d) = w_+$ for every $v \in \mathbb{Z}^d$ for which $y(v) = +$, and  $x(v + [-m,m]^d) = w_-$ for every $v \in \mathbb{Z}^d$ for which $y(v) = +$. We may do this by Lemma~\ref{dstarplacement} since $R_1 > 2k_m + 2m$, and so each distinct pair $v + [-m,m]^d$, $v' + [-m,m]^d$ where we place $w_+$ and $w_-$ are distance at least $2k_m$ apart. From now on, we will call these occurrences of $w_+$ and $w_-$ in $x$ ``intentional placements.'' Then, for every $t \in \mathbb{Z}^d$ for which $(n+2k_n)t + [1,n]^d$ is a distance of at least $k_n$ from each of the intentional placements within $x$, we define 
$x((n+2k_n)t + [1,n]^d) = w$. Again, we may use Lemma~\ref{dstarplacement} to define $x \in X$ with the desired occurrences of $w$ at the desired locations, since all of the translates of $[1,n]^d$ on which we are placing $w$ are distance of at least $2k_n$ apart. Also note that since we only placed these copies of $w$ at translates of $[1,n]^d$ which are distance at least $k_n$ from all intentional placements, no letter in any intentional placement is changed during this step. It is obvious that $\phi(x)$ agrees with $y$ on all of the nonzero symbols in $y$, so it remains to show that $(\phi(x))(v) = 0$ at all $v$ for which $y(v) = 0$. 

For this purpose, consider any $v \in \mathbb{Z}^d$ at which $y(v) = 0$, meaning that $x(v + [-m,m]^d)$ is not an intentional placement. We assume for a contradiction that $x(v + [-m,m]^d) \in \{w_+, w_-\}$.
Then we consider two cases. First, assume that $x(v + [-m,m]^d)$ does not even overlap an intentional placement. In this case, consider the cube $S$ of the form $(n+2k_n)t + [1,n]^d$ whose center is closest to the center $v$ of $v + [-m,m]^d$. Clearly the distance between $v$ and the center of $S$ is less than or equal to $\frac{n+2k_n}{2}$. Since 
$m > n + 2k_n$, this means that $S$ is contained within $v + [-m,m]^d$ and is distance at least $m - n - k_n$ from the boundary of $v + [-m,m]^d$. Therefore $S$ also has distance of at least $m - n - k_n$ from the closest intentional placement. 
This is greater than $k_n$ since $m > n + 2k_n$, and so by definition of $x$, we have that $x(S) = w$. Since $S \subset v + [-m,m]^d$ and $w_+, w_- \in L(X')$, this means that $x(v + [-m,m]^d) \notin \{w_+, w_-\}$, a contradiction.

We now deal with the case where $x(v + [-m,m]^d)$ does overlap an intentional placement. We first note that since $R_1 > 2k_m + 2m$, if $x(v + [-m,m]^d)$ overlaps an intentional placement, then every other intentional placement is distance at least $2k_m$ from $x(v + [-m,m]^d)$. Given this, define the unique 
$u \in \mathbb{Z}^d$ with $\|u\|_{\infty} < m$ so that $x(u + v + [-m,m]^d)$ is an intentional placement. (Note that $u \neq 0$.) Since $w_+$ and $w_-$ have the marker property, we have that 
$\|u\|_{\infty} > 2n + 4k_n$. Then regardless of $u$, $(v + [-m,m]^d) \setminus (u + v + [-m,m]^d)$ contains a translate of $[1, 2n + 4k_n]^d$, which we denote by $T$. As in the previous paragraph, if we define $S$ to be the cube of the form $(n+2k_n)t + [1,n]^d$ whose center is closest to the center of $T$, then $S$ is contained within $T$ and has distance at least $\frac{2n + 4k_n}{2} - n - k_n > k_n$ from the boundary of $T$, and therefore distance of more than $k_n$ from the intentional placement $x(u + v + [-m,m]^d)$. It is also a distance of at least $2k_m > k_n$ from all other intentional placements by the observation made above, since $S$ is contained within $v + [-m,m]^d$. Therefore, by construction of $x$, we have $x(S) = w$, and again, since $S \subset v + [-m,m]^d$, we see that $x(v + [-m,m]^d) \notin \{w_+, w_-\}$, a contradiction. We've now dealt with all possible cases, and so 
$x(v + [-m,m]^d) \notin \{w_+, w_-\}$. This implies that $(\phi(x))(v) = 0$, and so $y = \phi(x)$. Since $y \in W_{R_1, R_2}$ was arbitrary, we've also shown that $\phi(X) = W_{R_1, R_2}$.

\end{proof}

\begin{proof}[Proof of Theorem~\ref{badfactor}]

Theorem~\ref{badfactor} is a consequence of Theorems~\ref{WRnotIE} and \ref{WRfactor}. For any $X$ with the D*-condition, define $R_1$ to be the $N'$ satisfying Theorem~\ref{WRfactor}, and then define $R_2$ to be $N = N(R_1)$ from Theorem~\ref{WRnotIE}. Then $W_{R_1, R_2}$ is a factor of $X$ which is not intrinsically ergodic, and it was shown in the proof of Theorem~\ref{WRnotIE} that $h(W_{R_1, R_2}) \geq \frac{\log 2}{(R_1 + 1)^d} > 0$. 

\end{proof}

\section{Proof of Theorem~\ref{goodfactors}}
\label{goodproof}

We first define the examples of Hochman that we'll use to prove Theorem~\ref{goodfactors}. Here we only briefly summarize the examples and prove a few technical facts about them; for a full treatment, see \cite{hochman}. 

For any positive integer $k\in\nz$, a $\zt$ \sft\ $X_k$ is defined with alphabet $\ag_k$ consisting of $32+k$ symbols, which are most conveniently thought of as square tiles of unit length. Thirty-two of the symbols of $\ag_k$ are defined by assigning a tile one of four colors and one of eight types of arrows. For reasons which will become clear soon, the colors should be thought of as representing the four directions NW (northwest), NE (northeast), SW (southwest), and SE (southeast). The eight types of arrows are four straight arrows in the four cardinal directions (up, down, right, left), and four ``corner'' arrows turning ninety degrees clockwise. The remaining $k$ tiles are called ``blanks'' and are labeled with an associated integer between $1$ and $k$ inclusive. The symbols of $\ag_k$ appear in Figure \ref{PavlovF1}.

\begin{figure}[ht]
\centering
\includegraphics[scale=0.6]{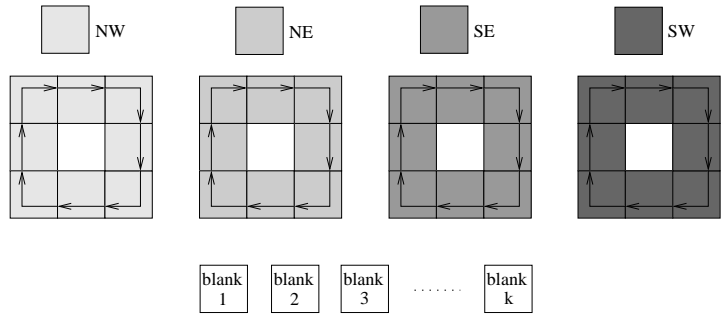}
\caption{The $32+k$ different symbols of $\ag_k$}
\label{PavlovF1}
\end{figure}

We define $X_k$ to be the SFT consisting of all points $x$ in $A_k^{\mathbb{Z}^2}$ for which every $2 \times 2$ subpattern of $x$ appears as a subpattern of the pattern in Figure~\ref{PavlovF2} for some choice of the labels of the blank symbols. Since for the purposes of these legal $2\times 2$ patterns, all blanks are considered indistinguishable, we may in the future suppress the labels of blank tiles, with the understanding that the $k$ blank tiles are completely interchangeable in elements of $X_k$.

\begin{figure}[ht]
\centering
\includegraphics[scale=0.7]{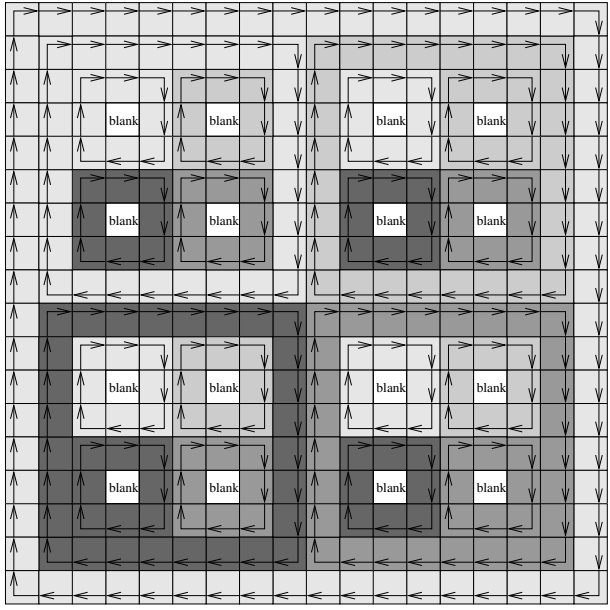}
\caption{Part of a point of $X_k$}
\label{PavlovF2}
\end{figure}

We inductively define valid patterns $P_n \in L_{[1,5 \cdot 2^n - 4]^d}(X_k)$ as follows: $P_0$ is a blank tile, and for any $n$, $P_{n+1}$ is created as follows. Create four square patterns by surrounding four copies of $P_n$ by a clockwise circuit of arrows colored in each of the four ways $\{NW, NE, SW, SE\}$. Then, concatenate those four square patterns into a larger square pattern $P_{n+1}$, where each of the four ``quadrants'' has arrows according to its location. As an example, the pattern appearing in Figure \ref{PavlovF2} is $P_3$ surrounded by a circuit of NW-colored arrows. The figure is also the upper-left quadrant of $P_4$. We call the pattern $P_n$ the \textbf{level-$n$ subsquare} of $X_k$. It is verified in \cite{hochman} that $P_n \in L(X_k)$ for all $n\in\nz$, and in fact we will see that in some sense most points in $X_k$ are built up out of the $P_n$. (We are still not concerning ourselves with the labeling of blank tiles, so each $P_n$ actually corresponds to many patterns in $L(X_k)$ depending on how the blanks are labeled.)

For any $\omega=(\omega_1,\omega_2,\ldots)\in\{NW, NE, SW, SE\}^{\mathbb{N}}$, we define $x_{\omega}\in X_k$ as follows. Begin by defining $x_{\omega}(0)$ to be a blank symbol. For any $n\geq 1$, assume that $x_{\omega}$ has already been defined on a square $B_n$ which is a translate of $[1, 5 \cdot 2^n - 4]^2$ and that $x_{\omega}(B_n) = P_n$. Then, surround $x_{\omega}(B_n)$ by a circuit of arrows colored by $\omega_n$. The resulting pattern appears as a subpattern of $P_{n+1}$ exactly once (in the proper quadrant), and so there is a unique way to extend $x_{\omega}$ to a square $B_{n+1}$ which is some translate of $[1, 5 \cdot 2^{n+1} - 4]^2$ so that $x_{\omega}(B_{n+1}):=P_{n+1}$. In addition, since each step includes surrounding by a circuit of arrows, it is easily checked by induction that $[-n,n]^2 \subseteq B_n$ for all $n\in\nz$, and so $\bigcup_{n=1}^{\infty} B_n=\zt$. Thus, $x_{\omega}$ is eventually defined on all of $\zt$ in this way.

It is proven in \cite{hochman} that every element $x\in X_k$ is of one of two types. Either $x=\sigma_t(x_{\omega})$ for some $t\in\zt$ and $\omega\in \{NW, NE, SW, SE\}^\nz$, or $x$ is what Hochman calls an ``exceptional point.'' There are five types of exceptional points, shown in Figure~\ref{PavlovF3} without colorings of the arrows. Every exceptional point is of one of these five types, up to a possible rotation and coloring allowed by the rules of $X_k$. Note that exceptional points contain no blanks, and so the sets of exceptional points in $X_k$ for any $k\in\nz$ coincide. We can therefore denote the set of exceptional points by $E$, regardless of $k$. 

\begin{figure}[ht]
\centering
\includegraphics[scale=0.25]{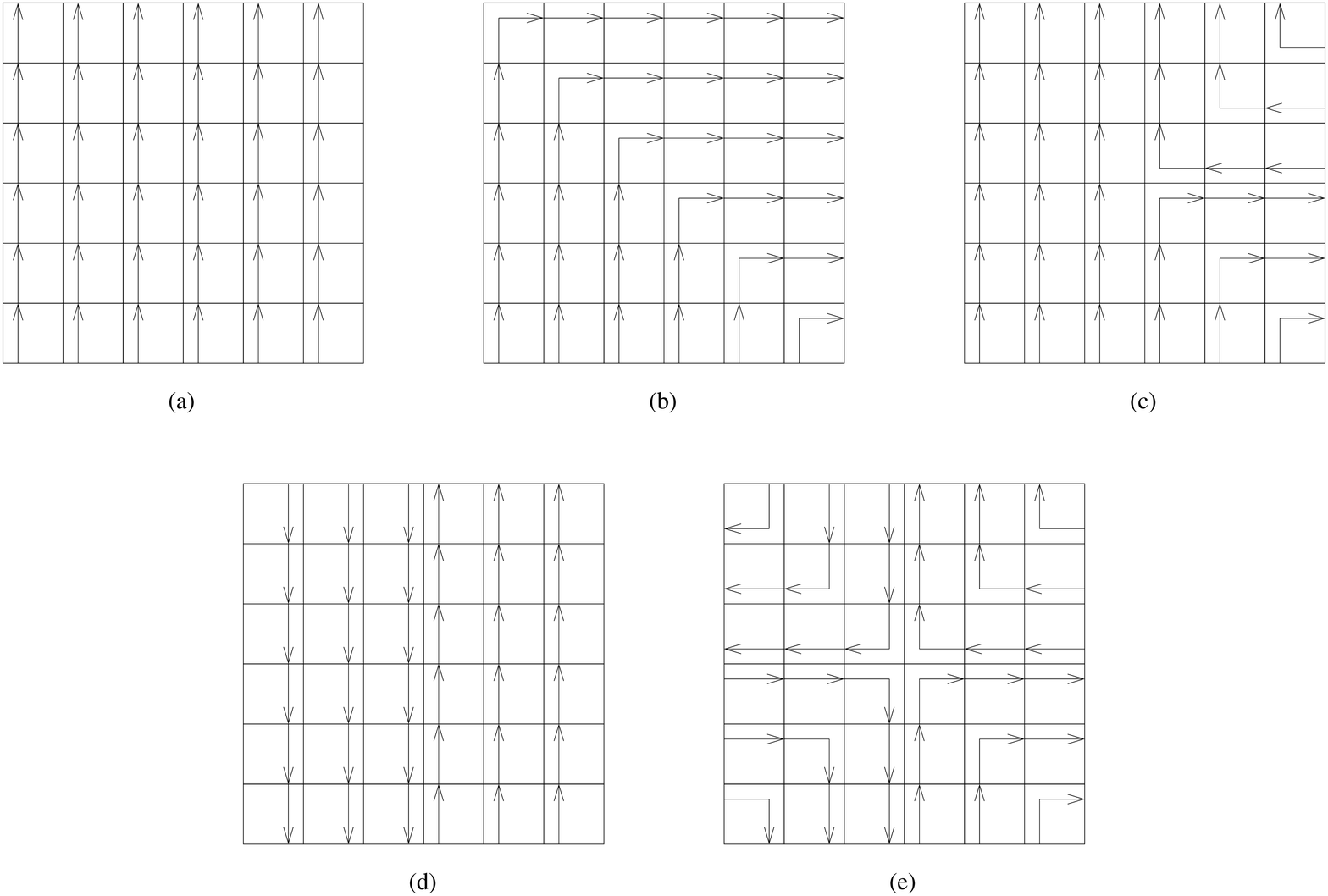}
\caption{Exceptional points of $X_k$ (uncolored)}
\label{PavlovF3}
\end{figure}

The following technical properties about $X_k$ follow fairly easily from its definition, but we present brief proofs for completeness.

\begin{proposition}\label{hochmanprops}

$X_k$ has the following properties:

\begin{enumerate}

\item Suppose $x \in X_k$ and $x(v)$ is a blank. For any $n \in \mathbb{N}$ and $u$ in $\mathbb{Z}^2$, if $v \in u+[-n,n]^2$, then there is a subset $S \subset \mathbb{Z}^2$ such that $u + [-n,n]^2 \subset S$ and $x(S)$ is a level-$2n$ subsquare in $x$.
\item For every $n$ and every nonexceptional $x \in X_k$, its upper limiting frequency of occurrences of level-$n$ subsquares is positive:
\[
\limsup_{k \rightarrow \infty} \frac{|[-k,k]^2 \cap S_x|}{(2k+1)^2} > 0,
\]
where $S_x$ is the set of locations of lower-left corners of level-$n$ subsquares in $x$.

\end{enumerate}

\end{proposition}

\begin{proof}

$ $\\

(1): We prove this statement by induction. The statement is trivial for $n = 0$, since level-$0$ subsquares are precisely blank symbols. Assume the statement for $n$. Then, if $x(v)$ is a blank and $v \in u +[-(n+1),n+1]^2$, then there exists $t \in [-1,1]^2$ so that $v \in u + t + [-n,n]^2$, and then $u + t + [-n,n]^2$ is contained entirely within a level-$2n$ subsquare by the inductive hypothesis. However, this level-$2n$ subsquare must be part of a level-$(2n+1)$ subsquare in $x$, in which the original level-$2n$ subsquare is surrounded by a circuit of arrows. Then $u + t + [-(n+1), n+1]^2$ is clearly contained entirely within that level-$(2n+1)$ subsquare. Finally, that level-$(2n+1)$ subsquare must be part of a level-$(2n+2)$ subsquare in $x$, in which it is surrounded by a circuit of arrows, and clearly $u + [-(n+1), n+1]^2$ is contained entirely within that level-$(2n+2)$ subsquare.\\

(2): First, note that for any $k \geq n$, the number of level-$n$ subsquares within any level-$k$ subsquare is $4^{k-n}$. Then, for any $\omega$ and any $k$, $x_{\omega}([-(5 \cdot 2^k - 4), 5 \cdot 2^k - 4]^2)$ contains a 
level-$k$ subsquare, and so contains at least $4^{k-n}$ level-$n$ subsquares. Therefore, the upper limiting frequency of level-$n$ subsquares in $x_{\omega}$ is at least
\[
\lim_{k \rightarrow \infty} \frac{4^{k-n}}{(10 \cdot 2^k - 7)^2} = \frac{1}{100 \cdot 4^n}.
\]

It is clear that translating a point does not change its upper limiting frequency of level-$n$ subsquares, and so this extends to all nonexceptional points $\sigma_t(x_{\omega})$.

\end{proof}

In the following definition, we introduce some auxiliary objects that will be useful in the proof of Theorem~\ref{goodfactors}.

\begin{definition}
 Let $n$ be in $\mathbb{N}$. Let $Y_{1,n}$ be the $\mathbb{Z}^2$ SFT obtained by replacing the lower-left corner of every level-$n$ subsquare in $X_1$ with the symbol $c$. (We note for reference that $Y_{1,n}$ is topologically conjugate to $X_1$, with conjugacy given by the $1$-block map that replaces each $c$ by the SW corner arrow.) For $m>1$, let $\mathcal{B}_m$ be the disjoint union of the alphabet of $X_1$ and the symbols $(c,i)$, where $i = 1, \dots, m$. Define the $1$-block map $\pi$ from $\mathcal{B}_m$ to the alphabet of $Y_{1,n}$ that acts as the identity on the alphabet of $X_1$ and sends each $(c,i)$ to $c$. We omit the dependence of $\pi$ on $m$ and $n$.) Let $Y_{m,n}$ be the $\mathbb{Z}^2$ SFT with alphabet $\mathcal{B}_m$ consisting of those points $x$ in $(\mathcal{B}_m)^{\mathbb{Z}^2}$ such that $\pi(x) \in Y_{1,n}$. 
\end{definition}

For notation, we let $\alpha_n$ be the supremum of upper limiting frequences of level-$n$ subsquares in $X_1$, which is positive by part (2) of Proposition~\ref{hochmanprops}. We now present some lemmas regarding the subshifts $Y_{m,n}$ that will be used in the proof of Theorem~\ref{goodfactors}. When $m$ and $n$ are fixed and $N > n$, a level-$N$ subsquare of $Y_{m,n}$ will be understood to mean a level-$N$ subquare of $X_1$ in which all lower-left corners of level-$n$ subsquares have been replaced by various $(c,i)$, $i = 1, \dots, m$. Let $\mathcal{S}_N$ denote the set of level-$N$ subsquares in $Y_{m,n}$. Note that $Y_{m,n}$ contains the same set $E$ of exceptional points as $X_1$. 

\begin{lemma} \label{YmnEntropy}
Suppose $m$ and $n$ are in $\mathbb{N}$. Then $h(Y_{m,n}) = \alpha_n \log m$.
\end{lemma}
\begin{proof}
 For a point $x$ in $X_1$, let $f_N(x)$ be the frequency of lower-left corners of level-$n$ subsquares in $x$ that appear in $[1,N]^2$. By definition of $\alpha_n$, there is a sequence $\{\epsilon_N\}_N$ tending to $0$ from above such that $\max_x f_N(x) \leq \alpha_n + \epsilon_N$ for all $N$. Also, there is a sequence of points $\{x_N\}_N$ in $X_1$ such that $\{f_N(x_N)\}_N$ tends to $\alpha_n$. Then 
\begin{equation*}
 m^{f_N(x_N) N^2} \leq |L_{[1,N]^2}(Y_{m,n})| \leq |L_{[1,N]^2}(X_1)| m^{(\alpha_n + \epsilon_N)N^2},
\end{equation*}
where the second inequality holds because each pattern in $Y_{m,n}$ is determined by a pattern $u$ in $X_1$ and a choice of symbol from $\{(c,i)\}_{i=1}^m$ at the lower-left corner of each level-$n$ subsquare in $u$.
Taking logs, dividing by $N^2$, and letting $N$ tend to infinity, we obtain
\begin{equation*}
 \alpha_n \log m \leq h(Y_{m,n}) \leq h(X_1) + \alpha_n \log m.
\end{equation*}
Since $h(X_1) = 0$, we see that $h(Y_{m,n}) = \alpha_n \log m$.
\end{proof}

\begin{lemma} \label{hochmanfactorUniqueness}
Suppose $m$ and $n$ are in $\mathbb{N}$. Then any measure $\mu$ in $\mathcal{M}(Y_{m,n})$ such that $\mu(E) = 0$ is uniquely determined by the values $\mu(u)$, for all level-$N$ subsquares $u$.
\end{lemma}
\begin{proof}
Let $\mu$ be in $\mathcal{M}(Y_{m,n})$ such that $\mu(E) = 0$. Then $\mu$ is uniquely determined by its values on the collection of cylinder sets $[w]$, where $w$ is in $L(Y_{m,n})$. Let $w$ be in $L(Y_{m,n})$, and suppose $w$ has shape $S$. Let us show that $\mu(w)$ is uniquely determined by the values of $\mu$ on the cylinder sets defined by level-$N$ subsquares for all $N$.

Observe that if $u$ is a level-$N_1$ subsquare containing a copy of $w$ and $v$ is a level-$N_2$ subsquare containing a copy of $w$, then one of the following holds: i) $[u]$ and $[v]$ are disjoint, ii) $[u]$ is contained in $[v]$, or iii) $[v]$ is contained in $[u]$. Now define $\mathcal{U}(w)$ to be the set of patterns $u$ such that $u$ is a level-$N$ subsquare for some $N$, $u$ contains a copy of $w$, and $u$ does not contain any level-$N'$ subsquare which contains a copy of $w$ for any $N' < N$.

By Proposition~\ref{hochmanprops}, for every non-exceptional point $x$ in $Y_{m,n}$ there exists $N$ such that $x(S)$ is entirely contained in a level-$N$ subsquare. Therefore the symmetric difference between $[w]$ and the disjoint union $\bigcup_{u \in \mathcal{U}(w)} [u]$ is contained in $E$. Since $\mu(E) = 0$, we have that $\mu(w)$ is the sum of the values $\mu(u)$, for $u$ in $\mathcal{U}(w)$.
Since $w$ was arbitrary, we conclude that the measure $\mu$ is uniquely determined as desired.
\end{proof}

\begin{lemma} \label{YmnIE}
 Suppose $m>1$ and $n$ is in $\mathbb{N}$. Then $Y_{m,n}$ is intrinsically ergodic.
\end{lemma}
\begin{proof}
 By Lemma~\ref{YmnEntropy}, we have that $h(Y_{m,n}) = \alpha_n \log m$. 

For now, we use the notation that $P_N$ is the unique level-$N$ subsquare in $Y_{1,n}$. For any invariant measure $\nu$ on $Y_{1,n}$ and $N_1 < N_2$, we have that
\begin{equation} \label{Eqn:GoHeels}
\nu(P_{N_1}) = 4^{N_2 - N_1} \nu(P_{N_2}),
\end{equation} 
since $P_{N_2}$ contains $4^{N_2-N_1}$ copies of $P_{N_1}$ and every $P_{N_1}$ is contained in a copy of $P_{N_2}$.\\

\textit{Claim:} there is a unique shift-invariant measure $\nu'$ on $Y_{1,n}$ such that $\nu'(P_n) \geq \alpha_n$, and $\nu'(E) = 0$. Consider any shift-invariant 
$\nu'$. Then by the ergodic theorem, there exists an integrable $g : Y_{1,n} \to \mathbb{R}$ such that
\[
\frac{1}{(2k+1)^2} \sum_{t \in [-k,k]^2} \chi_{[P_n]}(\sigma_t(x)) \underset{\nu'-\rm{a.e.}}{\longrightarrow} g,
\]
where $\int g \ d\nu' = \int \chi_{[P_n]} \ d\nu' = \nu'(P_n)$. Clearly $g = 0$ on all exceptional points, and is bounded from above by $\alpha_n$ due to its definition as the maximal upper limiting frequency of occurences of $P_n$ in points of $X_1$. Therefore, if $\nu'(P_n) \geq \alpha_n$, then $\nu'(E) = 0$. Since $\nu'(P_N)$ is determined by (\ref{Eqn:GoHeels}) for each $N$, we have that $\nu'$ is uniquely determined by Lemma~\ref{hochmanfactorUniqueness}. Thus, there is at most one measure with the desired properties.

Let us now show that there is at least one such measure. By definition of $\alpha_n$, there exists a sequence $\{x_N\}_N$ in $Y_{1,n}$ such that the frequency of $P_n$ in $x_N([-N,N]^2)$ tends to $\alpha_n$. Let
\begin{equation*}
\nu_N = \frac{1}{(2N+1)^2} \sum_{p \in [-N,N]^2} \delta_{\sigma_p(x_N)},
\end{equation*}
and let $\nu'$ be any subsequential (weak$^*$) limit of the sequence $\{\nu_N\}_N$. Then $\nu'$ is an invariant measure on $Y_{1,n}$ such that $\nu'(P_n) = \lim_N \nu_N(P_n)$, which is $\alpha_n$ by our choice of $\{x_N\}_N$.  Combining this fact with the result of the previous paragraph, we obtain that there is exactly one shift-invariant measure $\nu'$ on $Y_{1,n}$ such that $\nu'(P_n) \geq \alpha_n$, and $\nu'(E) = 0$, which establishes the claim.\\



Recall that $\pi : Y_{m,n} \to Y_{1,n}$ is the $1$-block factor map that projects $(c,i)$ to $c$ and acts as the identity on all other symbols. For a pattern $w$ in $Y_{1,n}$, we let $\pi^{-1}(w)$ denote the set of patterns $b$ in $Y_{m,n}$ such that $\pi(b) = w$. 

Let $\mu'$ be the measure in $\mathcal{M}(Y_{m,n})$ such that for each pattern $b$ in $\pi^{-1}(w)$, we have
\begin{equation} \label{DefOfMu}
 \mu'(b) = \nu'(w) m^{-F(w)},
\end{equation}
where $F(w)$ is the number of occurrences of the symbol $c$ in $w$. In words, $\mu'$ is the measure that projects to $\nu'$ under $\pi$ and, conditioned on $\pi^{-1}(w)$, independently gives uniform probability to each of the symbols $(c,i)$, for $i = 1, \dots, m$. 


Let $\mu$ be any measure in $\mathcal{M}(Y_{m,n})$ such that $\mu \neq \mu'$. We will show that $h(\mu) < \alpha_n \log m = h(Y_{m,n})$, and hence $\mu$ is not a measure of maximal entropy for $Y_{m,n}$. Let $\nu = \pi \mu$. For any measure $\tau$ on a subshift, recall that $H_{\tau,N}$ denotes the entropy of $\tau$ with respect to the partition into patterns of shape $[1,N]^2$. Using conditional probabilities and logarithmic rules, we have
\begin{align}
\begin{split} \label{Eqn:Ent1}
 H_{\mu,N} & = \sum_{b \in L_{[1,N]^2}(Y_{m,n})} -\mu(b) \log \mu(b) \\
           & = \sum_{w \in L_{[1,N]^2}(Y_{1,n})} \sum_{b \in \pi^{-1}(w)} - \mu \bigl( \pi^{-1}(w) \bigr) \mu \bigl( b \mid \pi^{-1}(w) \bigr) \log \Big(\mu \bigl( \pi^{-1}(w) \bigr) \mu \bigl( b \mid \pi^{-1}(w) \bigr)\Big) \\
					 & = \sum_{w \in L_{[1,N]^2}(Y_{1,n})} - \mu( \pi^{-1}(w) ) \log \mu( \pi^{-1}(w) ) \sum_{b \in \pi^{-1}(w)} \mu( b \mid \pi^{-1}(w) ) \\
                                              & \quad \quad + \sum_{w \in L_{[1,N]^2}(Y_{1,n})}  \mu( \pi^{-1}(w) ) \sum_{b \in \pi^{-1}(w)} - \mu( b \mid \pi^{-1}(w) ) \log \mu( b \mid \pi^{-1}(w))\\
					 & = \sum_{w \in L_{[1,N]^2}(Y_{1,n})} - \nu( w ) \log \nu( w )  + \sum_{w \in L_{[1,N]^2}(Y_{1,n})}  \nu( w ) H( \mu( \cdot \mid \pi^{-1}(w) ) ) \\
           & = H_{\nu,N} + \sum_{w \in L_{[1,N]^2}(Y_{1,n})} \nu(w) H( \mu( \cdot \mid \pi^{-1}(w) ) ).
\end{split}
\end{align}

Let $N_0 > n$ be in $\mathbb{N}$.
Let $\mu( \cdot \mid \mathcal{S}_{N_0})$ be the conditional probability measure on $\mathcal{S}_{N_0}$ induced by $\mu$: for $b$ in $\mathcal{S}_{N_0}$, let
\begin{equation*}
 \mu( b \mid \mathcal{S}_{N_0} ) = \frac{\mu(b)}{\sum_{b' \in \mathcal{S}_{N_0}} \mu(b')}.
\end{equation*}
Let $h'$ be the entropy of $\mu( \cdot \mid \mathcal{S}_{N_0})$, i.e. $h' = H(\mu( \cdot \mid \mathcal{S}_{N_0}))$. 

For a pattern $w$ in $L_{[1,N]^2}(Y_{1,n})$, we define a partition $\mathcal{P}(w)$ of $[1,N]^2$ as follows. Let $\mathcal{P}_1(w)$ consist of all $S$ for which $w(S)$ is a level-$N_0$ subsquare. Let $\mathcal{P}_2(w)$ consist of all $\{t\}$ so that $w(t) = c$, but this occurrence of $c$ is not contained in level-$N_0$ subsquare within $w$. Finally, let $\mathcal{P}_3(w)$ consist of $\{t\}$ for all $t$ such that $t$ is not contained in any level-$N_0$ subsquare within $w$ and $w(t) \neq c$. Then define $\mathcal{P}(w) = \mathcal{P}_1(w) \cup \mathcal{P}_2(w) \cup \mathcal{P}_3(w)$, and note that $\mathcal{P}(w)$ is a partition of $[1,N]^2$. For $C$ in $\mathcal{P}(w)$, recall that $\mu_C( \cdot \mid \pi^{-1}(w) )$ denotes the projection of $\mu( \cdot \mid \pi^{-1}(w) )$ onto the coordinates in $C$.  Then Proposition \ref{EntropyProdUB} gives
\begin{align}
\begin{split} \label{Dook}
H( \mu( \cdot \mid \pi^{-1}(w)) ) & \leq \sum_{C \in \mathcal{P}(w)} H( \mu_C( \cdot \mid \pi^{-1}(w)))\\
 & = \sum_{C \in \mathcal{P}_1(w)} H( \mu_C( \cdot \mid \pi^{-1}(w) ) ) \\
 & \quad \quad  + \sum_{\{t\} \in \mathcal{P}_2(w)} H( \mu_{\{t\}}( \cdot \mid \pi^{-1}(w) ) ) \\
 & \quad \quad + \sum_{\{t\} \in \mathcal{P}_3(w)} H( \mu_{\{t\}} ( \cdot \mid \pi^{-1}(w) ) ).
\end{split}
\end{align}

We now bound each of the three terms in the right-most expression of (\ref{Dook}).
For $C$ in $\mathcal{P}_1(w)$, we have that $\mu_C( \cdot \mid \pi^{-1}(w) ) = \mu( \cdot \mid \mathcal{S}_{N_0} )$. For notation, let $F_{N_0}(w)$ be the number of level-$N_0$ subsquares completely contained in $w$. Then by definition of $h'$, we have
\begin{equation} \label{DookBd1}
 \sum_{C \in \mathcal{P}_1(w)} H( \mu_C( \cdot \mid \pi^{-1}(w))) = |\mathcal{P}_1(w)| \cdot h' = F_{N_0}(w) \cdot h'.
\end{equation}
Now let $F_{\partial}(w)$ be the number of occurrences of the symbol $c$ in $w$ that appear within $5\cdot 2^{N_0}-4$ of the boundary of $[1,N]^2$; clearly $F_{\partial(w)} \leq 4 (5\cdot 2^{N_0}-4) N$. Note that for any singleton $\{t\}$ in $\mathcal{P}_2(w)$, we have that $t$ lies within $5 \cdot 2^{N_0}-4$ of the boundary of $[1,N]^2$.
Therefore, by Proposition~\ref{EntropyCardUB}, we see that
\begin{equation} \label{DookBd2}
 \sum_{\{t\} \in \mathcal{P}_2(w)} H( \mu_{\{t\}}( \cdot \mid \pi^{-1}(w) ) ) \leq |\mathcal{P}_2(w)| \log m \leq  F_{\partial}(w) \log m.
\end{equation}
Finally, for $\{t\}$ in $\mathcal{P}_3(w)$, we have that $H( \mu_{\{t\}}( \cdot \mid \pi^{-1}(w) ) ) = 0$, since all patterns in $\pi^{-1}(w)$ have the same symbol at location $t$.
Now combining (\ref{Dook}) - (\ref{DookBd2}), we see that
\begin{equation} \label{Rockies}
H( \mu( \cdot \mid \pi^{-1}(w) ) ) \leq F_{N_0}(w) \cdot h' + F_{\partial}(w) \log m.
\end{equation}
Combining (\ref{Eqn:Ent1}) and (\ref{Rockies}), we obtain
\begin{align}
\begin{split} \label{EntEqn}
 H_{\mu,N} & =  H_{\nu,N} + \sum_{w \in L_{[1,N]^2}(Y_{1,n})} \nu(w) H( \mu( \cdot \mid \pi^{-1}(w) ))\\
           & \leq H_{\nu,N} + \sum_{w \in L_{[1,N]^2}(Y_{1,n})} \nu(w) \biggl( F_{N_0}(w) \cdot h' + F_{\partial}(w) \log m \biggr).
\end{split}
\end{align}
Let $I_{N_0}(x)$ be the indicator function of the set of $x$ in $Y_{1,n}$ such that $x([1,5 \cdot 2^{N_0}-4]^2) = P_{N_0}$. Dividing by $N^2$ in (\ref{EntEqn}) and re-writing, we have that
\begin{align}
\begin{split} \label{Eqn:Ent5}
 \frac{1}{N^2} H_{\mu,N} & \leq \frac{1}{N^2} H_{\nu,N} +  h' \sum_{w \in L_{[1,N]^2]}(X_1)} \nu(w) \frac{F_{N_0}(w)}{N^2} + \frac{4 (5\cdot 2^{N_0}-4) N}{N^2} \log m \\
 & \leq \frac{1}{N^2} H_{\nu,N} +  h' \int \frac{1}{N^2} \sum_{t \in [1,N]^2} I_{N_0}(\sigma_t(x)) \, d\nu(x) + \frac{4 (5\cdot 2^{N_0}-4) }{N} \log m \\
 & = \frac{1}{N^2} H_{\nu,N} +  h' \nu(P_{N_0}) + \frac{4 (5\cdot 2^{N_0}-4) }{N} \log m.
\end{split}
\end{align}
Letting $N$ tend to infinity in (\ref{Eqn:Ent5}) and using that $h(\nu) = 0$, we see that
\begin{equation} \label{Eqn:MainPoint}
h(\mu) \leq h(\nu) + \nu(P_{N_0}) \cdot h' = \nu(P_{N_0}) \cdot h'.
\end{equation}

By (\ref{Eqn:GoHeels}) and the above claim, we have
\begin{equation} \label{Eqn:MainPoint2}
 \nu(P_{N_0}) = 4^{-N_0+n} \nu(P_n) \leq 4^{-N_0+n} \alpha_n,
\end{equation}
with equality if and only if $\nu = \nu'$.
Furthermore, by definition of $h'$ and Proposition \ref{EntropyCardUB}, we have that
\begin{equation} \label{Eqn:MainPoint3}
 h' \leq \log |\mathcal{S}_{N_0}| = 4^{N_0-n} \log m,
\end{equation}
with equality if and only if $\mu( \cdot \mid \mathcal{S}_{N_0} )$ is uniform on $\mathcal{S}_{N_0}$.
Combining (\ref{Eqn:MainPoint}), (\ref{Eqn:MainPoint2}), and (\ref{Eqn:MainPoint3}), we have that
\begin{equation} \label{Eqn:MainPoint4}
 h(\mu) \leq \alpha_n \log m,
\end{equation}
with equality only if $\nu = \nu'$ and $\mu( \cdot \mid \mathcal{S}_{N_0} )$ is uniform on $\mathcal{S}_{N_0}$, for all $N_0 > n$.

Hence, if $\nu \neq \nu'$, then (\ref{Eqn:MainPoint4}) gives that $h(\mu) < \alpha_n \log m$.
If $\nu = \nu'$, then Proposition \ref{hochmanfactorUniqueness}, (\ref{DefOfMu}), and the fact that $\mu \neq \mu'$ together imply that there exists $N_0 > n$ such that $\mu( \cdot \mid \mathcal{S}_{N_0} )$ is not uniform on $\mathcal{S}_{N_0}$. In this case, we again obtain that $h(\mu) < \alpha_n \log m$ by (\ref{Eqn:MainPoint4}).
Taken together, these cases  show that if $\mu \neq \mu'$, then $h(\mu) < \alpha_n \log m = h( Y_{m,n})$.
As $Y_{m,n}$ is expansive and therefore has a measure of maximal entropy, we obtain that $\mu'$ is the unique measure of maximal entropy on $Y_{m,n}$.
Thus, $Y_{m,n}$ is intrinsically ergodic.
\end{proof}

\begin{lemma} \label{hochmanfactorFactors}
 Suppose $\varphi : X \to Z$ is a factor map with radius $n$ such that the cardinality of the set of $\varphi$-images of level-$2n$ subsquares is $m$. Then there exist factor maps $\psi_1 : X \to Y_{m,2n}$ and $\psi_2 : Y_{m,2n} \to Z$ such that $\varphi = \psi_2 \circ \psi_1$.
\end{lemma}
\begin{proof}
 Suppose the set of $\varphi$-images of level-$2n$ subsquares is $\{w_i\}_{i=1}^m$. For $x$ in $X$, define $\psi_1(x)$ as follows. If $x(v)$ is the lower-left corner of level-$2n$ subsquare whose $\varphi$-image is $w_i$, then let $(\psi_1(x))(v) = (c,i)$. If $x(v)$ is a blank, then let $(\psi_1(x))(v)$ be the blank symbol in $Y_{m,n}$. Otherwise, let $\psi_1(x)(v) = x(v)$. One may easily verify that $\psi_1$ is a factor map.

For $x$ in $Y_{m,n}$, define $\psi_2(x)$ as follows. Suppose $v + [-n,n]^2$ is contained in a level-$2n$ subsquare $x(u+[1,5 \cdot 2^{2n}-4]^2)$ whose lower-left corner is labeled $(c,i)$. Suppose $v = u+p$. Then let $(\psi_2)(x)(v) = w_i(p)$. If $v + [-n,n]^2$ is not contained in a level-$2n$ subsquare in $x$, then let $(\psi_2(x))(v) = \varphi(\tilde{x})(v)$, where $\tilde{x}$ is a point in $X$ obtained by replacing all symbols $(c,i)$ in $x$ with the SW arrow symbol and arbitrarily choosing blanks. This map is well-defined, since $\varphi$ has radius $n$, and by part (1) of Proposition~\ref{hochmanprops}, if $v + [-n,n]^2$ is not contained in any level-$2n$ subsquare, then $\tilde{x}(v+[-n,n]^2)$ cannot contain any blanks. Furthermore, one may easily check that $\psi_2$ is a shift-commuting continuous map and that $\varphi = \psi_2 \circ \psi_1$, implying that $\psi_2$ is surjective and therefore also a factor map.
\end{proof}

\begin{lemma} \label{hochmanfactorEntropy}
 Suppose $\varphi : X \to Z$ is a factor map with radius $n$ such that the cardinality of the set of $\varphi$-images of level-$2n$ subsquares is $m$. Then $h(Z) = \alpha_{2n} \log m$. 
\end{lemma}
\begin{proof}
By Lemma \ref{hochmanfactorFactors}, there exist factor maps $\psi_1 : X \to Y_{m,2n}$ and $\psi_2 : Y_{m,2n} \to Z$ such that $\varphi = \psi_2 \circ \psi_1$. By the fact that entropy is non-increasing under factor maps and Lemma \ref{YmnEntropy}, we have $h(Z) \leq h(Y_{m,2n}) = \alpha_{2n} \log m$. 

Let us now check the reverse inequality. By definition of $\alpha_{2n}$, there exists a sequence $\{x_N\}_N$ of points in $Y_{1,n}$ such that the frequency of lower-left corners of level-$2n$ subsquares contained in $x_N([1,N]^2)$, denoted here by $F_N(x_N)$, tends to $\alpha_{2n}$.  For each $x_N$, we have that $\pi^{-1}(x_N([1,N]^2))$ contains $m^{F_N(x_N) N^2}$ patterns, each of which has a distinct image under $\psi_2$. Hence
\begin{equation*}
\frac{1}{N^2} \log |L_{[1,N]^2}(Z)| \geq \frac{1}{N^2} \log \biggl( m^{F_N(x_N) N^2} \biggr) =  F_N(x_N) \log m.
\end{equation*}
Letting $N$ tend to infinity gives that $h(Z) \geq \alpha_{2n} \log m$, as desired.
\end{proof}

\begin{proof}[Proof of Theorem~\ref{goodfactors}]

By part (2) of Proposition~\ref{hochmanprops}, there exist points of $X_1$ with positive upper limiting frequency of blank symbols. Then clearly $h(X_k)$ is at least $\log k$ times that positive number, and so can be made arbitrarily large by choosing $k$ large. We choose any $k > 1$, and, for convenience, we denote $X_k$ simply by $X$. 

Consider any factor map $\phi$ on $X$ with radius $n$ where $h(\phi(X)) > 0$. By Lemma~\ref{hochmanfactorEntropy}, there exist at least $2$ distinct $\phi$-images of level-$2n$ subsquares; call them $w_1, \ldots, w_m$, with $m > 1$.  By Lemma~\ref{hochmanfactorFactors}, there exist factor maps $\psi_1 : X \to Y_{m,2n}$ and $\psi_2 : Y_{m,2n} \to \phi(X)$ such that $\phi = \psi_2 \circ \psi_1$. By Lemmas \ref{YmnEntropy} and \ref{YmnIE}, $Y_{m,2n}$ has entropy $\alpha_{2n} \log m$, and there is a unique invariant probability measure $\mu$ on $Y_{m,2n}$ with entropy $h(\mu) = \alpha_{2n} \log m$.

By Lemma~\ref{hochmanfactorEntropy}, we have that $h(\phi(X)) = \alpha_{2n} \log m$. Since entropy cannot increase under a factor map and $\mu'$ is the only measure on $Y_{m,2n}$ with $h(\mu') \geq \alpha_{2n} \log m$, we see that there is at most one measure on $\phi(X)$ (namely the push-forward $\psi_2 \mu'$) with entropy $\alpha_{2n} \log m$. Using the fact that any subshift has a measure of maximal entropy, we conclude that there is a unique measure of maximal entropy on $\phi(X)$.

\end{proof}

We note that $X_k$ has factors which are not intrinsically ergodic; for instance, if $\psi$ removes the labels of blanks, then $X_1 = \psi(X_k)$ is zero entropy, but not uniquely ergodic (due to the existence of exceptional points), and all of its measures are trivially measures of maximal entropy. If one wants an example where all factors, including zero entropy ones, are intrinsically ergodic, then as mentioned in the introduction, any uniquely ergodic subshift would suffice. There exist uniquely ergodic $\mathbb{Z}^d$ shifts of finite type (see, for instance, \cite{mozes} and \cite{robinson}), but it was shown in \cite{radin} that such SFTs must have topological entropy $0$. 

To summarize, we have examples of $\mathbb{Z}^2$ SFTs for which every positive entropy factor is intrinsically ergodic, and examples of $\mathbb{Z}^2$ SFTs for which every factor is zero entropy and uniquely ergodic. An ideal example would combine the properties of these two, i.e. it would have positive entropy and it would have the property that every factor (including zero entropy factors) is uniquely ergodic. However, we do not yet know of such an example.


\begin{thebibliography}{99}


\bibitem{bowen}
R. Bowen,
{\em Some systems with unique equilibrium states}, 
Math. Syst. Theory {\bf 8} (1974), 193--202.

\bibitem{boyle}
M. Boyle, 
{\em Open problems in symbolic dynamics}, 
Contemp. Math. {\bf 469} (2008), 69--118.

\bibitem{BPS}
M. Boyle, R. Pavlov and M. Schraudner, 
{\em Multidimensional sofic shifts without separation and their factors}, 
Trans. Amer. Math. Soc. {\bf 362} (2010), 4617--4653.

\bibitem{BS}
R. Burton and J. Steif, 
{\em Non-uniqueness of measures of maximal entropy for subshifts of finite type}, 
Ergodic Theory Dynam. Systems {\bf 14} (1994), no. 2, 213--235.

\bibitem{CT}
V. Climenhaga and D. Thompson, 
{\em Intrinsic ergodicity beyond specification: $\beta$-shifts, $S$-gap shifts, and their factors},
Israel J. Math. {\bf 192} (2012), no. 2, 785--817.

\bibitem{CoverThomas}
T. Cover and J. Thomas,
{\em Elements of information theory.} 
John Wiley \& Sons (2012).


\bibitem{hochman}
M. Hochman, 
{\em On the automorphism groups of multidimensional SFTs}, 
Ergodic Theory Dynam. Systems {\bf 30} (2010), no. 3, 809--840.

\bibitem{LG}
J. Lebowitz and G. Gallavotti,
{\em Phase transitions in binary lattice gases},
J. Math. Phys. {\bf 12} (1971), no. 7, 1129--1133.

\bibitem{LM}
D. Lind and B. Marcus, 
{\em Introduction to symbolic dynamics and coding.} 
Cambridge University Press, Cambridge (1995).

\bibitem{misiurewicz}
M. Misiurewicz, 
{\em A short proof of the variational principle for a $\mathbb{Z}_+^n$-action on a compact space},
Asterisque, \textbf{40} (1975), 147--157.

\bibitem{mozes}
S. Mozes, 
{\em Tilings, substitution systems and dynamical systems generated by them}, 
J. Analyse Math. {\bf 53} (1989), 139--186.

\bibitem{parry}
W. Parry, 
{\em Intrinsic Markov chains}, 
Trans. Amer. Math. Soc. {\bf 112} (1964), 55--66.

\bibitem{radin}
C. Radin,
{\em Disordered Ground States of Classical Lattice Models}, 
Rev. Math. Phys. {\bf 3} (1991), 125--135.

\bibitem{robinson}
R.M. Robinson, 
{\em Undecidability and non-periodicity of tilings of the plane}, 
Inventiones Math. {\bf 12} (1971), 177--209.


\bibitem{ruelle}
D. Ruelle,
{\em Thermodynamic Formalism.}
Cambridge University Press, Cambridge (1995).    

\bibitem{weiss}
B. Weiss,
{\em Intrinsically ergodic systems}, 
Bull. Amer. Math. Soc. {\bf 76} (1970), no. 6, 1266--1269.

\bibitem{WS}
S.G. Whittington and C.E. Soteros,
{\em Lattice animals: rigorous results and wild guesses},
Disorder in physical systems pp. 323--335, 
Oxford Science Publications, Oxford Univ. Press, New York, 1990.

\bibitem{WR}
B. Widom and J.S. Rowlinson, 
{\em New Model for the Study of Liquid-Vapor Phase Transitions},
J. Chem. Phys. {\bf 52} (1970), 1670--1684.

\end{thebibliography}
\end{document}